\documentclass[10pt]{amsart}

\usepackage{amssymb,amsmath,amsthm}
\usepackage{graphicx,a4wide}

\theoremstyle{plain}

\usepackage{graphicx,tikz}
\newtheorem{theorem}{Theorem}

\newtheorem*{lemma}{Lemma}

\theoremstyle{definition}

\theoremstyle{remark}

\DeclareMathOperator\Ai{Ai}
\DeclareMathOperator\Bi{Bi}

\begin{document}

\title[]{Quantitative Projections in the \\Sturm Oscillation Theorem}

\author[]{Stefan Steinerberger}
\address{Department of Mathematics, Yale University}
\email{stefan.steinerberger@yale.edu}

\begin{abstract}
There is $c_{} > 0$ such that for all $f \in C[0,\pi]$ with at most $d-1$ roots inside $(0,\pi)$
$$  \sum_{1 \leq n \leq d}{ \left| \left\langle f,  \sin\left( n x\right) \right\rangle \right|} \geq
\kappa^{-\kappa^2 \log{\kappa}}\|f\|_{L^2} \qquad \mbox{where} \quad \kappa = \frac{c_{} \| \nabla f\|_{L^2}}{\|f\|_{L^2}}.$$
This quantifies the Sturm-Hurwitz Theorem and connects a purely topological condition (number of roots) to the Fourier spectrum. It is also one of few estimates on Fourier coefficients from \textit{below}. 
The result holds more generally for eigenfunctions of regular Sturm-Liouville problems
$$ - (p(x) y'(x))' + q(x) y(x) = \lambda w(x) y(x) \qquad \mbox{on}~(a,b).$$
Sturm-Liouville theory shows the existence of a sequence of solutions $(\phi_n)_{n=1}^{\infty}$
that form an orthogonal basis of $L^2(a,b)$ with respect to $w(x)dx$.  Sturm himself proved that if $f:(a,b) \rightarrow \mathbb{R}$ is a finite linear combinations of $\phi_n$
having $d-1$ roots inside $(a,b)$, then $f$ cannot be orthogonal to $A = \mbox{span}\left\{\phi_1, \dots, \phi_{d}\right\}$. We prove
a lower bound on the size of the projection $\| \pi_A f\|_{}$. \end{abstract}

\maketitle

\section{Introduction}
\subsection{Introduction.}
This paper is concerned with the Sturm-Liouville problem 
\begin{align*}
 - (p(x) y'(x))' + q(x) y(x) &= \lambda w(x) y(x) \qquad \mbox{on an interval}~(a,b)\\
p(a)y'(x) - \alpha y(a) &= 0\\
p(b) y'(x) + \beta y(b) &= 0,
\end{align*}
where $p, q, w$ are smooth, $p, w > 0$ are uniformly bounded away from 0, $q \geq 0$ is nonnegative, $\lambda \geq 0$ is the eigenvalue and $\alpha, \beta$ are nonnegative constants.
The cases $\alpha = \infty$ and $\beta = \infty$ are permitted and should be understood as Dirichlet boundary conditions.
 If $p(x) \equiv 1 \equiv w(x)$ and $\alpha = \beta = 0$, we recover the classical eigenvalue problem for the Schr\"odinger operator
$$ H = -\frac{d^2}{dx^2} + q(x)$$
as a special case. The study of these objects, Sturm-Liouville theory, dates back to seminal papers of Sturm and Liouville from 1836  \cite{liouville, sturm, sturm2}.
The main purpose of our paper is to show that the strong form of the Sturm Oscillation Theorem \cite{sturm} can be made quantitative. Sturm proved that there exists a discrete set of parameter $(\lambda_n)_{n=1}^{\infty}$ (the eigenvalues of the Sturm-Liouville
operator) and an associated sequence of solutions $(\phi_n)_{n=1}^{\infty}$ that form an orthogonal basis in $L^2(a,b)$. 
He further established a structure statement for the solutions $\phi_n$.

\begin{quote}
\textbf{(Weak) Sturm Oscillation Theorem.} $\phi_n$ has $n-1$ roots in $(a,b)$.
\end{quote} 
Many other structural properties are known: one that is also commonly found in textbooks is that the roots of consecutive solutions are interlacing. 
 However, both Sturm and Liouville originally proved a \textit{much} stronger result (Sturm being the first to establish the result, Liouville then gave a different proof). That stronger result is not very well known and
reads as follows.

\begin{quote}
\textbf{Sturm Oscillation Theorem.} For any integers $m \leq n$ and any set of coefficients $a_m, a_{m+1}, \dots, a_n$ such that not all of them are 0, the function
$$ \sum_{k=m}^{n}{a_k \phi_k} \qquad \mbox{has at least}~m-1~\mbox{and at most}~n-1~\mbox{roots in}~(a,b).$$
\end{quote} 
The  case $\phi_n(x) = \sin{nx}$ is sometimes known as the Sturm-Hurwitz theorem
after being stated by Hurwitz \cite{hurwitz} (who explicitly mentions Sturm and also deals with the case $n=\infty$). The Sturm-Hurwitz theorem has had substantual impact in partial differential equations, see \S 2.1. 
A recent paper by B\'erard \& Helffer \cite{berard} chronicles the history of the result (and how it was forgotten) and gives an accessible, clear and modern description of the original proofs. An engrossing historical description of the development of Sturm-Liouville theory is given by L\"utzen \cite{lutzen}; we quickly quote verbatim from \cite[\S 2]{lutzen}
since it describes a historic event that should be more widely known.
\begin{quote}
In 1833 both \textsc{Sturm} and \textsc{Liouville} and their common friend \textsc{J. M. C. Duhamel} applied for the seat vacated by the death of \textsc{A. M. Legendre}. A fourth applicant was \textsc{G. Libri-Carucci} [...] On March $18^{th}$, \textsc{Libri} was elected with 37 votes against \textsc{Duhamel's} 16 and \textsc{Liouville} 1. Nobody voted for \textsc{Sturm}. The next opportunity was offered after the
death of \textsc{Ampere} in the summer of 1836. [...] Three weeks before the election [...] \textsc{Liouville} presented a paper to the Academy in which he praised \textsc{Sturm's} two memoires on the Sturm-Liouville theory as ranking with the best works of \textsc{Lagrange}. Supporting a rival in this way was rather unusual in the competitive Parisian academic circles, and it must have been shocking when on the day of the election, December $5^{th}$, \textsc{Liouville} and \textsc{Duhamel} withdrew their candidacies to secure the seat for their friend. \textsc{Sturm} was elected with an overwhelming majority. (L\"utzen, \cite[\S 2]{lutzen})
\end{quote}

\subsection{Our result.}

One particular implication of the (strong) Sturm Oscillation Theorem will be the following: if $f \in C[a,b]$ has at most $d-1$ roots (counted without multiplicity) in $(a,b)$, then the function $f$ cannot be orthogonal to the subspace
$$ A = \mbox{span}\left\{\phi_1, \dots, \phi_d\right\} \subset L^2(a,b).$$
 It becomes natural to ask how large the projection onto that subspace is, i.e. to determine lower bounds on the size of $\pi_{A} f$. This question, while of intrinsic interest, is also naturally
related to problems related to the behavior of partial differential equations, we refer to \cite{gal, stein2}.
It is easy to see that some additional condition is necessary: take $f$ to be a mollification of the function $\delta_{x} - \delta_{x+\varepsilon}$, where $\delta_x$ denotes
the Dirac delta in $x$ and $\varepsilon > 0$ is arbitrary (and $a < x < x+ \varepsilon < b$). This function is highly localized and has mean value zero, there is no control on its spatial scale. In particular, the integral over the product of $f$ and \textit{any} smooth function can be arbitrarily small: the eigenfunctions of the Sturm-Liouville operator are certainly smooth, therefore no uniform bounds on $\| \pi_A f\|_{L^2}$ are possible. However, we also observe that functions $f$ so constructed necessarily have to have $\| \nabla f \|_{L^2}$ large: the main purpose of this paper is to note that once we account for this fact and impose some smoothness on $f$, uniform results are indeed possible.

\subsection{Organization.} \S 2 discusses the main results. We start by describing the main result for the special case $- y''(x) = \lambda y(x)$ with Dirichlet boundary conditions $[0, \pi]$ (which turns out to be a fairly typical case) and demonstrate how 
it implies a refined Sturm-Hurwitz theorem. \S 2.2. discusses a second case involving Airy functions to show that the result truly holds a greater level of generality. \S 2.3. states the main result,
\S 2.4. discusses how this can be applied to some related problems arising in the study of integral operators. \S 3 discusses a curious combinatorial Lemma that plays a subtantial role in the proof 
of the main result. \S 4 gives a proof of the main result; a proof of Theorem 1, the special case $-y''(x) = \lambda y(x)$, follows immediately by replacing the general Sturm-Liouville eigenfunctions $\phi_n(x)$ with $\sin{(nx)}$
(and various steps in the proof simplify). 

\section{Main Results}

\subsection{Trigonometric Functions.}We first, for illustrative purposes, state the main result in two special cases where everything can be made explicit.
The first special case deals with
$$ - y''(x) = \lambda  y(x) \qquad \mbox{on}~(0,\pi) \qquad \mbox{and}~y(0) = y(\pi) = 0.$$
The eigenvalues $(\phi_n)_{n=1}^{\infty}$ are given by $\lambda_n = n^2$ and the eigenfunctions by $\phi_n(x) = \sin{( n x)}$. Sturm's Oscillation Theorem implies that if
$$ f(x) = \sum_{n=1}^{N}{a_n \sin{(n x)}} \quad \mbox{has}~d-1~\mbox{roots in}~(0,\pi)\mbox{, then} \quad \sum_{n=1}^{d}{|a_n|} > 0$$
This result is sometimes called the Sturm-Hurwitz Theorem. The case $N=\infty$ turns out to be admissible but the conclusion has never been strenghtened. However, there has been substantial interest in the statement itself. Polya \cite{pol} emphasizes the connection to the heat equation; indeed, for a diffusion process the number of roots is necessarily nonincreasing and this can be used as a proof of the statement by taking time $t \rightarrow \infty$ (this line of reasoning originates with Sturm and Liouville).  Kolmogorov, Petrovskii \& Piskunov \cite{kpp} rediscovered the principle in their original work on the KPP equation. Tabachnikov \cite{taba}, following Blaschke \cite{blaschke}, observed that it generalizes a famous result in the differential geometry of plane curves, the four vertex theorem of Mukhopadhyaya \cite{muk} and Kneser \cite{kne} (see also Arnol'd \cite{arnold}).
Arnol'd discusses the Sturm-Hurwitz theorem at length \cite{arn2} in his third 1997 Lecture at the Fields institute. The survey of Galaktionov \& Harwin \cite{gal} discusses its importance for parabolic partial differential equations.\\

 There has been great interest in generalizing this principle. Eremenko \& Novikov \cite{eremenko, erem2} established a beautiful continuous Sturm-Hurwitz theorem (functions whose Fourier transform is supported a fixed distance from the origin oscillate at least at a certain rate) that was originally conjectured by Logan \cite{logan}. Generalizations to higher dimensions have proven to be of substantial difficulty: even the number of nodal domains of a single eigenfunction of 
$-\Delta$ remains poorly understood, we refer to the seminal papers of Courant \cite{cour} and Pleijel \cite{plei}, more recent papers are \cite{ber, bour, stein}. The author showed \cite{stein1} that any linear combination of Laplacian eigenfunctions whose eigenvalue exceeds a certain limit $\lambda$ has to vanish somewhat often but the scaling in that result is likely not optimal: a better understanding of the higher-dimensional case is much desired. 

\begin{theorem} Assume $f \in C[0,\pi]$ has $d-1$ roots in $(0,\pi)$. Then, for some universal $c > 0$,
$$  \sum_{n = 1}^{d}{ \left| \left\langle f,  \sin\left( n x\right) \right\rangle \right|} \geq \kappa^{-\kappa^2 \log{\kappa}} \|f\|_{L^2} \qquad \mbox{where} ~\kappa = c_{} \frac{\| \nabla f\|_{L^2([0,\pi])}}{\|f\|_{L^2([0,\pi])}}.$$
\end{theorem}
We observe that the inequality provides a \textit{lower} bound on the size of certain Fourier coefficients, there seem to be very few inequalities of this type \cite{stein2}; moreover, this lower bound is motivated by a purely 'topological' object, the number of roots. Several additional remarks are in order.

\begin{quote} \textbf{Remarks.}
\begin{enumerate}
\item  It is not important that $\kappa$ has that precise form, the term 
$$ c\| \nabla f\|^{}_{L^2} \|f\|_{L^2}^{-1} \quad \mbox{can be replaced by} \quad c_s\| \nabla^s f\|^{1/s}_{L^2} \|f\|_{L^2}^{-1/s}$$
for any $s > 0$ where $c_s$ is a constant only depending on $s$. Any coercive pseudodifferential operator with the right homogeneity could be used.
\item We have no reason to assume that the scaling $\kappa^{-\kappa^2 \log{\kappa}}$ is optimal.
\item We also prove a version of the result for sign changes instead of roots (\S 2.3).
\item There is natural analogue of Theorem 1 with $\cos{(nx)}$ instead of $\sin{(nx)}$ that can be obtained simply by replacing Dirichlet with Neumann conditions.
\end{enumerate}
\end{quote}

\subsection{Airy Functions.}
The second example is meant to demonstrate that the underlying technique can truly be applied in the much broader context of general Sturm-Liouville problems. Let us consider the following ordinary differential equation on $[0,1]$ 
$$ - y''(x) + (1+x) y(x) = \lambda  y(x) \qquad \mbox{satisfying}~~y(0) = y(1) = 0.$$
This equation is still, purposefully, fairly simple so as to allow for a solution that can be written down in closed form: the general solution of the equation is given by
$$ y(x) = a \Ai(1 - \lambda + x) + b \Bi(1 - \lambda + x),$$
where $\Ai$ and $\Bi$ denote the two linearly independent solution of Airy's equation $y''= x y$.
If both boundary conditions can be satisfied, then necessarily
$$ \det \begin{pmatrix} \Ai(1 - \lambda) & \Ai(2 - \lambda) \\ \Bi(1 - \lambda) & \Bi(2 - \lambda) \end{pmatrix} = 0.$$

\begin{center}
\begin{figure}[h!]
\begin{tikzpicture}[scale=1]
\draw [thick, ->] (-4,-0.53) -- (4,-0.53);
\node at (0,0) {\includegraphics[width=0.5\textwidth]{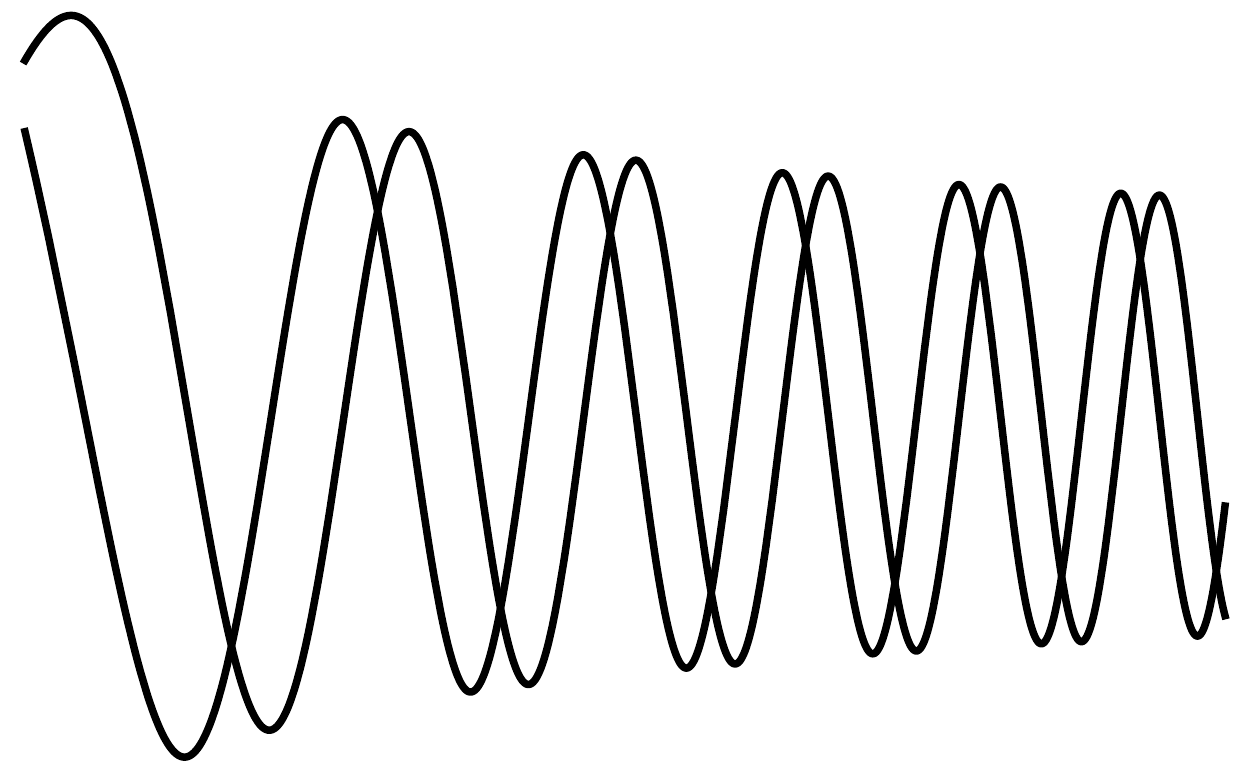}};
\end{tikzpicture}
\caption{$\Ai(1-x)$ and $\Bi(1-x)$ in the interval $(1.5, 15)$. There are infinitely many intervals of length 1 on which linear combinations of these two functions
give rise to a solution of the ODE. }
\end{figure}
\end{center}

This equation is satisfied for an increasing sequence of eigenvalues 
$$ \lambda_1 = 11.3685\dots, \lambda_2 = 40.9787\dots, \lambda_3 = 90.3266\dots $$
The eigenvalues grow like $\lambda_n \sim c n^2$ with a constant that could be explicitly computed (using the Weyl asymptotic \cite[Theorem 5.25]{teschl}). The associated eigenfunctions 
$$\phi_n(x) = a_n \Ai(1 - \lambda_n + x) + b_n \Bi(1 - \lambda_n + x)$$
are defined, up to sign, by $L^2-$normalization and behave quite similarly to the classical trigonometric functions. This is not
surprising since the equation can be interpreted as the classical ordinary differential equation for trigonometric functions subjected to a lower order perturbation.
Elementary facts about $\Ai$ and $\Bi$ are sufficient to deduce the same result with the same proof: if $f\in C[0,1]$ has $d-1$ roots (or, more generally, sign changes) in $(0,1)$, then we have, for some universal $c>0$,
$$  \sum_{n = 1}^{d}{ \left| \left\langle f,  \phi_n(x) \right\rangle \right|} \geq \kappa^{-\kappa^2 \log{\kappa}} \|f\|_{L^2([0,1])},~ \mbox{where} ~\kappa = c_{} \frac{\| \nabla f\|_{L^2([0,1])}}{\|f\|_{L^2([0,1])}}.$$
We see that this case is more or less identical to the case discussed above. This is not a coincidence and the observed behavior as well as the function $\kappa^{-\kappa^2 \log{\kappa}}$ are
actually generic. Again, there is no reason to assume that this function  $\kappa^{-\kappa^2 \log{\kappa}}$ is sharp for any of these cases.

\subsection{The general case.} We will now state the main result in a somewhat informal fashion. We assume that we are given a Sturm-Liouville problem satisfying the properties
described above. 

\begin{theorem} Let $f \in C[a,b]$ and assume $f$ has $d-1$ roots in $(a,b)$.  Then
$$  \sum_{n = 1}^{d}{ \left| \left\langle f,  \phi_n \right\rangle \right|} \geq g(\kappa) \|f\|_{L^2([0,\pi])} \qquad \mbox{where} ~\kappa = c_{} \frac{\| \nabla f\|_{L^2([0,\pi])}}{\|f\|_{L^2([0,\pi])}}$$
and $g:\mathbb{R}_{>0} \rightarrow \mathbb{R}_{>0}$ is a function that only depends on (and can be explicitly constructed out of) the eigenvalues $\lambda_n$, the sequence $(\|\phi_n\|_{L^{\infty}})_{n \in \mathbb{N}}$ and the sequence
$$ h(n) = \min_{1 \leq k \leq n} \left\{ \left| \phi_k(x)\right|: \phi_k'(x) = 0\right\}$$
if we are dealing with Dirichlet or Neumann boundary conditions. Under general boundary conditions, we define $h$(n) as the smallest maximal value of $|\phi_k|$, $k \leq n$, within a nodal domain.

\end{theorem}
We observe that all three quantities are rather well-behaved and close to universal (up to constants).  In particular, we know that $\lambda_n \sim n^2$ (where the implicit constant depends on the actual Sturm-Liouville problem), we expect that $\|\phi_n\|_{L^{\infty}} \sim 1$ and that $h(n) \sim 1$. All three conditions are satisfied in both examples that were considered above. There are various types of improvements that seem feasible. In particular, we can replace 'roots' by 'sign changes' -- the proof of this statement is merely a minor and straightforward variation on our proof (we chose the version stated in Theorem 2 primarily for its simplicity and similarity to the classical Sturm Oscillation Theorem; the proof of Theorem 2 addresses the necessary modifications).

\subsection{Related results.} The results in this paper are related (and partially inspired) by the following type of inverse problem: let $f \in C^{\infty}_{c}(0,1)$ and let $T$ be an
integral operator 
$$ Tf = \int_{\mathbb{R}}^{}{k(x,y)f(y) dy}.$$
Under which conditions on the kernel $k(x,y)$ is it possible to obtain results of the following type: the quantity $\|Tf\|_{L^2(J)}$, $J$ being a generic interval, say $J=(2,3)$, is not very small unless $f$ oscillates very rapidly on $(0,1)$? This is still somewhat vague but every possible way of making it precise (say, choice of norms etc.) would be interesting.
This question is rather fundamental but also severely ill-posed since integral operators are compact: in particular, there are necessarily functions $f$ that do not even oscillate all that rapidly for which
$Tf$ is indeed quite tiny (see Fig. 2). Very little seems to be known. Al-Aifari, Pierce and the author \cite{rima} have shown that for the Hilbert transform $H$, the kernel
being $k(x,y) = (x-y)^{-1}$ and for disjoint intervals $I,J \subset \mathbb{R}$,
$$  \|H f\|_{L^2(J)} \geq c_1 \exp{\left(-c_2\frac{ \|f_x\|_{L^2(I)}}{\|f\|_{L^2(I)}}\right)} \| f \|_{L^2(I)},$$
where the constants $c_1, c_2$ depend only on the intervals $I,J$. This result is sharp up to constants in the following sense: if $c_2$ is too small, then there is
an infinite sequence of orthonormal functions for which the inequality fails. 
We refer to Fig. 2 for an example of a rather nice, smooth function  (this example is taken from \cite{rima}) which shows that some type of strong decay in these estimates will indeed be necessary.
Sharp inequalities of a similar type have since been established 
for the Laplace transform the Fourier transform and the Riesz transform \cite{angie, roy}. So far, nothing seems to be known for general integral operators (or even integral operators of
convolution type with arbitrarily strong restrictions on the kernel $k(x-y)$). Currently, there are two different types of approaches to this question. We start by describing the first one:
\begin{enumerate} 
\item rewrite the quantity for which we would like to obtain lower bounds as  
$$ \|Tf\|_{L^2}^2 = \left\langle Tf, Tf \right\rangle = \left\langle T^* T f, f\right\rangle$$
\item the operator $T^* T$ is self-adjoint and we can apply the spectral theorem; the eigenvalues $\lambda_n$ of $T^* T$ converge very quickly to
zero but there is some quantitative control on the decay
\item there exists a (higher-order) differential operator $D$ such that
$$ T^*TD = D T^* T \qquad \mbox{(Slepian's miracle)}$$
\item this implies that the eigenfunctions of $T^*T$ and $D$ coincide and we may work with eigenfunctions of $D$ instead: if the function does not oscillate rapidly, then $\|Df\|_{L^2}$ is small implying that a spectral expansion of $f$ into eigenfunctions has a nontrivial component at low frequency. This, in turn, ensures that $\|Tf\|$ cannot be too small. 
\end{enumerate}

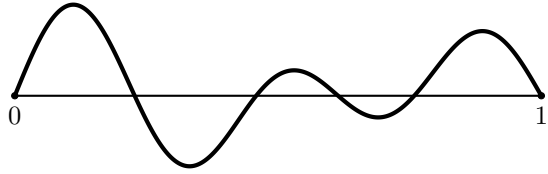
\begin{figure}[h!]
\begin{center}
\begin{tikzpicture}[xscale=7,yscale=1.1*(7/9)]
\draw [ultra thick, domain=0:1, samples = 300] plot (\x, {-0.15269*sin(2*pi*\x r)    +    0.4830*sin(3*pi*\x r)  +  0.3084*sin(4*pi*\x r)  + 0.80509*sin(5*pi*\x r)}  );
\draw [thick, domain=0:1] plot (\x, {0}  );
\filldraw (0,0) ellipse (0.006cm and 0.048cm);
\node at (0,-0.3) {0};
\filldraw (1,0) ellipse (0.006cm and 0.048cm);
\node at (1,-0.3) {1};
\end{tikzpicture}
\caption{$f:[0,1] \rightarrow \mathbb{R}$ with $\|Hf\|_{L^2[2,3]} \sim 10^{-3}\|f\|_{L^2[0,1]}$} 
\end{center}
\end{figure}

Step (3) is problematic since, to the best of our knowledge, it can only be carried out in fairly specialized circumstances (this limits the applicability of the method). 
The second and completely different approach is due to A. R\"uland \cite{angie} and based on PDE techniques:  the Hilbert transform can be realized as the tangential limit of a harmonic equation in higher dimensions for which it is possible to use propagation of smallness results. This approach can be applied to the classical Hilbert transform and Riesz transforms in higher dimensions: just like the other method,
it seems unlikely that it can be extended to much more general perators.
The result discussed in this paper has immediate
applications to step (4): if $D$ happens to be a Sturm-Liouville operator, then the results in this paper suggest the existence of lower bounds based on the number of sign
changes alone. More precisely, if $f \in C[a,b]$ has $d-1$ roots in $(a,b)$, then we obtain quantitative control on the inner product of
$f$ with the first $d$ eigenfunctions of the operator and can thus bound the size of $\|T f\|$ in terms of $\|\nabla f\|_{}, \|f\|$ and the eigenvalues of $T^* T$.

\section{A Combinatorial Lemma} 
 The proof has one amusing ingredient of a combinatorial flavor. Let
$ a \in \mathbb{R}_{\geq 0}^{n}$ be a vector with nonnegative entries and consider the sequence of vectors $a_1, a_2, \dots$ defined via
$$ a_{\ell} = (a_1, a_2/2^{\ell}, a_3/3^{\ell}, \dots, a_n/n^{\ell}).$$

We see that each entry is undergoing exponential decay but at different (exponential) rates. If $a \neq 0$, then there must exist some $\ell$ such that $a_{\ell}$ has essentially
one entry that is much larger than all the other entries combined. By taking $\ell \rightarrow \infty$, it is easy to see that the first nonzero entry has the desired property
but the value of $\ell$ necessary for this to occur may be arbitrarily large (depending on how small the first nonzero entry is). However, if we are happy with merely finding
\textit{some} entry that is much larger than the rest combined (not necessarily the first nonzero one that is guaranteed to dominate in the limit), then we can guaranteed that this is possible for a fairly small value of $\ell$ 
where 'fairly small' does \textit{not} depend on the entries of $a$.

\begin{lemma} Let $ 0 \neq(a_1, \dots, a_n) \in \mathbb{R}_{\geq 0}^n$ and $0< b_1 < \dots < b_n$. 
For every $\varepsilon \in (0,1/2)$, there
exists 
$$\ell \leq  \frac{n \log{(9n/\varepsilon^2)}}{ \log\left( \min_{2 \leq i \leq n}{ \frac{b_{i}}{b_{i-1}}} \right)} $$ 
and $ k \in \left\{1, \dots, n\right\}$ such that
$$  \sum_{i=1}^{n}{ \frac{a_i}{b_i^{\ell}}} \leq (1+\varepsilon) \frac{a_k}{b_k^{\ell}}.$$
\end{lemma}
Since we can always apply the result to the list of numbers $ \left( a_1/b_1^L, a_2/b_2^L, \dots, a_n/b_n^L\right),$ the Lemma says that the set of integers $\ell$ for which there exists such a suitable $k$ is infinite and has bounded gaps.  As mentioned above, all sufficiently large numbers have the desired property (but 'sufficiently large' will depend on the coefficients which would be insufficient for our problem). 
We quickly illustrate how the Lemma is going to be applied in the special case of trigonometric functions. Let the function $f: (0,\pi) \rightarrow \mathbb{R}$ be given by
$$ f(x) = \sum_{k=1}^{n}{a_k \sin{(kx)}} \qquad \mbox{and define} \qquad f_{\ell} = \frac{(-\Delta)^{-\ell}f_{}}{\left\| (-\Delta)^{-\ell}f_{} \right\|_{L^{\infty}}}.$$
If $f \not\equiv 0$, then the sequence $f_{\ell}$ converges to a pure sine frequency as $\ell \rightarrow \infty$. 
Our Lemma implies that $f_{\ell}$ will already be uniformly $\varepsilon-$close to \textit{some} pure sine frequency (not necessarily the one arising in the limit) for at least one
$$\ell \lesssim n^2 \log{\left( \frac{n}{\varepsilon^2}\right)}.$$
\begin{figure}[h!]
\begin{center}
\begin{tikzpicture}[xscale=7,yscale=0.15)]
\draw [thick, domain=0:1, samples = 300] plot (\x, {0.01*sin(pi*\x r)    +    5*sin(3*pi*\x r)  +  3*sin(5*pi*\x r)  + 4*sin(12*pi*\x r)  + 5*sin(20*pi*\x r)}  );
\draw [ultra thick, domain=0:1, samples = 300] plot (\x, {(0.01*sin(pi*\x r)    +    5*sin(3*pi*\x r)/9^2  +  3*sin(5*pi*\x r)/25^2 )*160  }  );
\draw [dashed, domain=0:1, samples = 300] plot (\x, {11*sin(pi*\x r)     }  );
\draw [thick, domain=0:1] plot (\x, {0}  );
\filldraw (0,0) ellipse (0.006cm and 0.048cm);
\node at (-0.03,-1.3) {0};
\filldraw (1,0) ellipse (0.006cm and 0.048cm);
\node at (1.03,-1) {$\pi$};
\end{tikzpicture}
\caption{The Combinatorial Lemma applied to trigonometric polyomials: for any trigonometric polynomial $f$ of degree $n$, the sequence $(-\Delta)^k f/ \|(-\Delta)^k f\|_{L^{\infty}}$ converges
to the lowest non-vanishing frequency ($f_{100}$, dashed) as $k \rightarrow \infty$. However, the sequence contains at least one element (here: $f_{2}$, bold) where $k \lesssim_{\varepsilon} n^2 \log{n}$ that is essentially a pure frequency with a small error term.}
\end{center}
\end{figure}
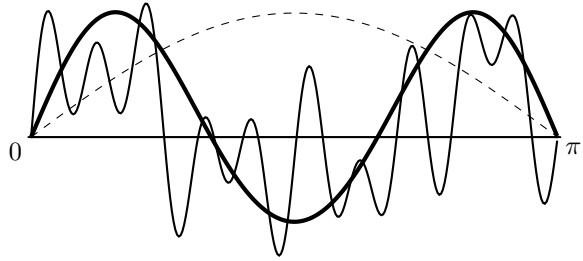

\begin{proof}[Proof of the Lemma.]
We start by observing that we can assume that
$$ \sum_{i=1}^{n}{a_i} \geq (1+\varepsilon)a_n \qquad \mbox{or, equivalently,} \quad a_n \leq \frac{1}{\varepsilon} \sum_{i=1}^{n-1}{a_i}$$
because otherwise we are already done ($k=n$ and $\ell = 0$). This shows us that at least some of the mass is distributed on the first $n-1$ coordinates.
The exponential decay then implies that by taking sufficiently large powers, we can ensure that that remaining mass grows by a disproportional
factor. Abbreviating
 $$ b = \max_{2 \leq i \leq n}{ \frac{b_{i-1}}{b_i} },$$
we see that
$$ \frac{a_n}{b_n^{\ell}} \leq   \frac{1}{b_n^{\ell}} \frac{1}{\varepsilon} \sum_{i=1}^{n-1}{a_i} \leq \frac{1}{\varepsilon} \sum_{i=1}^{n-1}{ \left(\frac{b_{i}}{b_n}\right)^{\ell}  \frac{a_i}{b_i^{\ell}}} \leq \frac{b^{\ell}}{\varepsilon} \sum_{i=1}^{n-1}{ \frac{a_i}{b_i^{\ell}}}.$$ 
This shows that, as soon as $b^{\ell} \leq \varepsilon^2/n$, we have that
$$ \frac{a_n}{b_n^{\ell}} \leq \frac{\varepsilon}{n} \sum_{i=1}^{n-1}{ \frac{a_i}{b_i^{\ell}}} \quad \mbox{and thus} \quad  \sum_{i=1}^{n}{ \frac{a_i}{b_i^{\ell}}} \leq \left(1 + \frac{\varepsilon}{n}\right)  \sum_{i=1}^{n-1}{ \frac{a_i}{b_i^{\ell}}}$$
We note that if these two inequalities is satisfied for one $\ell$, then it is automatically satisfied by all larger integers as well. 
We now fix the smallest $\ell$ such that $b^{\ell} \leq \varepsilon^2/n$ and repeat the process on on the new set
$$ \left\{\frac{a_1}{b_1^{\ell}}, \frac{a_2}{b_2^{\ell}}, \dots, \frac{a_{n-1}}{b_{n-1}^{\ell}} \right\}.$$
We will continue the process if
$$ \sum_{i=1}^{n-1}{\frac{a_i}{b_i^{\ell}}} \geq (1+\varepsilon) \frac{a_{n-1}}{b_{n-1}^{\ell}} \qquad \mbox{or, equivalently,} \quad  \frac{a_{n-1}}{b_{n-1}^{\ell}} \leq \frac{1}{\varepsilon} \sum_{i=1}^{n-1}{\frac{a_i}{b_i^{\ell}}}$$
and stop if that inequality is not satisfied. Suppose we continue. Then there exists $\ell_2$ such that
$$ \frac{a_{n-1}}{b_{n-1}^{\ell +\ell_2}}  \leq \frac{b^{\ell_2}}{\varepsilon} \sum_{i=1}^{n-2}{ \frac{a_i}{b_i^{\ell + \ell_2}}}.$$ 
We pick $\ell_2$ as the smallest integer for which $b^{\ell_2} \leq \varepsilon^2/n$ (in particular, $\ell_2 = \ell$) and repeat the process.
Ultimately, the process terminates and we are left with an index $1 \leq m \leq n$ and an integer $L \in \mathbb{N}$ such that 
$$  \sum_{k=1}^{m}{\frac{a_k}{b_k^L}} \leq  (1 + \varepsilon) \frac{a_m}{b_m^L} .$$
By construction of the process, we have
\begin{align*}
 \sum_{k=1}^{n}{\frac{a_k}{b_k^L}} \leq \left(1 + \frac{\varepsilon}{n}\right) \sum_{k=1}^{n-1}{\frac{a_k}{b_k^L}} \leq  \left(1 + \frac{\varepsilon}{n}\right)^2 \sum_{k=1}^{n-2}{\frac{a_k}{b_k^L}} \leq  \left(1 + \frac{\varepsilon}{n}\right)^{n-m} \sum_{k=1}^{m}{\frac{a_k}{b_k^L}} \leq e^{\varepsilon}  \sum_{k=1}^{m}{\frac{a_k}{b_k^L}}
\end{align*}
and therefore, for $0 < \varepsilon < 1/2$,
$$  \sum_{k=1}^{n}{\frac{a_k}{b_k^L}} \leq e^{\varepsilon}  \sum_{k=1}^{m}{\frac{a_k}{b_k^L}} \leq e^{\varepsilon}  (1 + \varepsilon) \frac{a_m}{b_m^L} \leq (1 + 3\varepsilon) \frac{a_m}{b_m^L}.$$
This is the desired statement.
It remains to understand how large $L$ can be: in the worst case, we have to run the scheme for $n$ steps, each step requiring to take a power large enough so that $b^{\ell} \leq \varepsilon^2/n$ implying
$$L \leq n \frac{\log{(n/\varepsilon^2)}}{\log{(1/b)}}.$$
\end{proof}

\section{Proof of Theorem 2}
This section gives the proof of the main result, a proof of Theorem 1 follows by specializing $\phi_n(x) = \sin{nx}$. 
We require one elementary Lemma stating that Sturm-Liouville eigenfunctions essentially behave as trigonometric functions in the sense of inducing an
equivalent Sovolev space. 

\begin{lemma} We have, up to absolute constants depending only on $p(x),q(x),w(x)$,
$$\| \nabla f \|_{L^2}^2 \sim \sum_{n=1}^{\infty}{\lambda_n \left| \left\langle f, \phi_n \right\rangle \right|^2}$$
for all functions $f \in C^2[a,b]$ that have at least one root.
\end{lemma}
\begin{proof} One direction follows immediately from $q(x) \geq 0$, the fact that $p(x) > 0$ is uniformly bounded away from 0, integration by parts and completeness of $(\phi_n)$ in $L^2(a,b)$
\begin{align*}  \int_{a}^{b}{ f'(x)^2 dx} &\leq \frac{1}{\min_{a \leq x \leq b}{p(x)}} \int_{a}^{b}{p(x) f'(x)^2 + q(x) f(x)^2 dx}\\ 
&=  \frac{1}{\min_{a \leq x \leq b}{p(x)}}\left\langle -(p(x)f'(x))' + q(x) f(x), f(x) \right\rangle\\
& =  \frac{1}{\min_{a \leq x \leq b}{p(x)}}\left\langle \sum_{n =1}^{\infty}{ \lambda_n \left\langle f, \phi_n\right\rangle \phi_n},  \sum_{n =1}^{\infty}{ \left\langle f, \phi_n \right\rangle \phi_n} \right\rangle\\ 
&=   \frac{1}{\min_{a \leq x \leq b}{p(x)}}\sum_{n =1}^{\infty}{ \lambda_n \left| \left\langle f, \phi_n \right\rangle \right|^2}.
\end{align*}
Suppose now that $f(x_0) = 0$. Then we can argue that
\begin{align*}
\int_{a}^{b}{q(x) f(x)^2 dx} &\leq \max_{a \leq x \leq b}{q(x)} \int_{a}^{b}{ \left( \int_{x_0}^{x}f'(y)dy\right)^2 dx} \\
&\leq  \max_{a \leq x \leq b}{q(x)}  (b-a) \int_{a}^{b}{ \left( \int_{\min\left\{ x_0, x\right\}}^{\max\left\{x_0, x\right\}}f'(y)^2 dy\right)   dx} \\
&\leq  \frac{\max_{a \leq x \leq b}{q(x)}}{\min_{a \leq x \leq b}{ p(x)}}  (b-a)^2 \int_{a}^{b}{ p(x) f'(x)^2 dx}.
\end{align*}
This implies the desired equivalence since
\begin{align*}
\int_{a}^{b}{p(x) f'(x)^2 + q(x) f(x)^2 dx} &\leq \left(1 + \frac{\max_{a \leq x \leq b}{q(x)}}{\min_{a \leq x \leq b}{ p(x)}}  (b-a)^2\right)\int_{a}^{b}{p(x) f'(x)^2dx}\\
&\leq  \left(1 + \frac{\max_{a \leq x \leq b}{q(x)}}{\min_{a \leq x \leq b}{ p(x)}}  (b-a)^2\right) \max_{a \leq x \leq b}{p(x)} \int_{a}^{b}{f'(x)^2 dx}.
\end{align*}
\end{proof}

\begin{proof}[Proof of Theorem 2] We now present the proof of Theorem 2. A proof of Theorem 1 is included as a special case by setting $\lambda_n = n^2$ and $\phi_n(x) = \sin{(nx)}$ (in that case, many of the steps can be considerably simplified). We start by noting that we can apply the heat equation
$$ u_t =  (p(x) y'(x))' - q(x) y(x)$$
for an arbitrary short amount of time; this increases the regularity of the function and is a continuous operation for the function that we consider since it is diagonalized by Sturm-Liouville eigenfunctions. We can thus assume that $f \in C^3[0, \pi]$.
  Let $f \in C^3[0,\pi]$ and assume $f$ has $d-1$ roots inside the interval $(0,\pi)$. Roots that are not necessarily sign changes, i.e. double roots or roots of higher order, are unstable under small perturbations. Since there are only finitely many, we can perform a slight mollification to remove them (indeed, this actually paves the way to a slightly stronger result where instead of counting roots we count roots that are also sign changes and we count them without multiplicity). This can be done in a way that changes $\| \nabla f\|_{L^2}/\|f\|_{L^2}$ by an arbitrarily small amount. We are interested in obtaining
$$ \mbox{lower bounds on} \qquad  \sum_{n = 1}^{d}{ \left| \left\langle f,  \phi_n \right\rangle \right|}$$
and this sum will undoubtedly be affected by the perturbation, however, we can again ensure that the quantity is perturbed an arbitrarily small amount: if we are interested in proving a lower bound $\delta$, as claimed in the statement, we can ensure that the perturbation changes the quantity by at most a factor of $\delta/100$ and still conclude the result with constant $99 \delta/100$. We can furthermore assume that $f'$ has a nonvanishing derivative in the places where it changes sign. A slightly different approach to this perturbation is given at the end of the paper. We shall abuse notation by using $f$ to denote this mollified function; in particular, it suffices to prove the statements for functions $f$ that have $d-1$ sign changes, no additional roots and a nonvanishing derivative in its roots.
We decompose the function $f$ as
\begin{align*}
 f = \sum_{n=1}^{\infty}{ \left\langle f, \phi_n \right\rangle \phi_n}.
\end{align*}
It is classical (see e.g. Titchmarsh \cite{titchmarsh}) that Sturm-Liouville expansion inherit essentially the same kind of regularity properties as Fourier series. In particular, for $f \in C^3$ the expansion converges uniformly for the function and its derivatives on any compact subset inside $(a,b)$. Various weaker conditions, $f \in C^{2+\varepsilon}$ or some bound on total variation would be enough.
 In particular, there exists $N_0 \in \mathbb{N}$ such that for all $N \geq N_0$
$$ g = \sum_{n=1}^{N}{ \left\langle f, \phi_n \right\rangle \phi_n} \qquad \mbox{has exactly}~d-1~\mbox{roots}.$$
After possibly further increasing $N$, we may assume that $\|g\|_{L^2} \geq \|f\|_{L^2}/2$. We now continue working with $g$ and decompose it into three different terms
$$ g = \sum_{n =1}^{N}{ \left\langle f, \phi_n \right\rangle \phi_n} = \sum_{n \leq d}^{}{ \left\langle f, \phi_n \right\rangle \phi_n}  + \sum_{d+1 \leq n \leq M}^{}{ \left\langle f, \phi_n \right\rangle \phi_n} +\sum_{M+1 \leq n \leq N}^{}{ \left\langle f, \phi_n \right\rangle \phi_n},$$
where $M$ will be chosen so large that the first two terms contains at least half of the $L^2-$mass of $g$. The next step is to show that the following choice of $M$ indeed has the desired property
$$ M = \frac{c \| \nabla f\|_{L^2}}{\|f\|_{L^2}}.$$
Here, the constant $c$ will depend on the constants in the Lemma showing equivalence of the classical and induced Sobolev space which themselves depend only on $p(x), q(x)$ and $w(x)$ in the Sturm-Liouville problem.
We use the Lemma, the symbol $\sim$ hiding implicit constants depending on the problem, to conclude 
\begin{align*}
\| \nabla f \|_{L^2}^2 &\sim \sum_{n=1}^{\infty}{\lambda_n \left| \left\langle f, \phi_n \right\rangle \right|^2} \geq  \sum_{n \geq M}^{}{\lambda_n \left| \left\langle f, \phi_n \right\rangle \right|^2} 
\geq   \lambda_{M}  \sum_{n \geq M}^{}{\left| \left\langle f, \phi_n \right\rangle \right|^2}.
\end{align*}
Using the Weyl asymptotic (see, for example, \cite[Theorem 5.25]{teschl})
$$ \lambda_n = \pi^2 \left( \int_{a}^{b}{ \sqrt{\frac{w(t)}{p(t)}} dt} \right)^{-2}n^2 + \mathcal{O}(n),$$
we can conclude that for a suitable choice of $c$ (depending on the implicit constants in the Lemma which in turn depend on the coefficients in the Sturm-Liouville problem) and the definition of $M$
$$ \sum_{n \geq M}^{}{\left| \left\langle g, \phi_n \right\rangle \right|^2}  \leq \sum_{n \geq M}^{}{\left| \left\langle f, \phi_n \right\rangle \right|^2} \leq \frac{\|f\|_{L^2}^2}{100} \leq \frac{\|g\|_{L^2}^2}{25}.$$
We now make use of an explicit monotonicity formula due to Sturm (nicely explained in \cite[Proposition 2.11]{berard}):
for every integer $\ell \geq 1$, the function
$$ g_{\ell} =  \sum_{n \leq d}^{}{ \left\langle g, \phi_n \right\rangle \lambda_n^{-\ell} \phi_n} + \sum_{d+1 \leq n \leq c_2M}^{}{ \left\langle g, \phi_n \right\rangle \lambda_n^{-\ell}  \phi_n} + \sum_{c_2 M + 1\leq n \leq N}^{}{ \left\langle g, \phi_n \right\rangle \lambda_n^{-\ell}  \phi_n}$$
has less roots (counted with multiplicity) in $(a,b)$ than $g$ and therefore at most $d-1$ roots (counted with multiplicity) since $f$, after perturbation, has $d-1$ roots each of which has multiplicity 1. The remainder of the argument will work as follows.
\begin{enumerate}
\item We show that, for a suitable choice of $c_2$ and all $\ell$ sufficiently large, the third sum is much smaller in $L^{\infty}$ than the second sum is in $L^2$.
\item We use the Combinatorial Lemma to conclude that for a suitable value of $\ell$, the second sum is essentially one pure frequency up to small errors in $L^{\infty}$.
\item Any pure frequency that could occur in the second sum has at least $d$ roots, the function itself has only $d-1$ roots. This cannot be changed by the third term which is too small.
\item This is therefore corrected by the first term; this term can thus not be arbitrarily small.
\end{enumerate}

 We start by estimating the third term via
\begin{align*}
  \left\| \sum_{c_2M + 1\leq n \leq N+1}^{}{ \left\langle g, \phi_n \right\rangle \lambda_n^{-\ell}  \phi_n} \right\|_{L^{\infty}}  &\leq  \left(\sup_{n \in \mathbb{N}}{ \|\phi_n\|_{L^{\infty}}}\right) \sum_{c_2M + 1\leq n \leq N+1}^{}{ \left| \left\langle g, \phi_n \right\rangle \right| \lambda_n^{-\ell}  } \\
&\leq  \left(\sup_{n \in \mathbb{N}}{ \|\phi_n\|_{L^{\infty}}}\right) \left( \sum_{ n \geq c_2M + 1}^{}{ \left| \left\langle g, \phi_n \right\rangle \right|^2  } \right)^{\frac{1}{2}}  \left( \sum_{ n \geq c_2M + 1}^{}{ \lambda_n^{-2\ell}  } \right)^{\frac{1}{2}}  \\
&\leq c_3\left(\sup_{n \in \mathbb{N}}{ \|\phi_n\|_{L^{\infty}}}\right) \frac{\|g\|_{L^2}}{ (\lambda_{c_2M})^{ \ell - \frac{1}{2}}},
\end{align*}
where the constant $c_3$ depends on the growth of the eigenvalues of $\lambda_n$ and the speed with which they approximate the Weyl-asymptotic. 
We now use the Combinatorial Lemma for the sum
$$  \sum_{d+1 \leq n \leq c_2M}^{}{ \left\langle g, \phi_n \right\rangle \lambda_n^{-\ell}  \phi_n}.$$
This requires us to understand the quantity $b = \min \lambda_{n+1}/\lambda_n$. We have no external information about this quantity and its size will enter as a factor into the function $g(\kappa)$ in the statement of Theorem 2. However, we have a pretty good understanding what happens generically and expect
$$ b(c_2 M) =  \min_{d+1 \leq n \leq c_2 M}{ \frac{\lambda_{n+1}}{\lambda_n}} \sim \frac{(c_2 M + 1)^2}{c_2^2 M^2} \sim 1 + \frac{1}{c_2 M}.$$
This, in particular, is true for the trigonometric functions; we assume that it remains valid and proceed with that assumption. In a setting where that assumption is incorrect, one has to introduce 
an explicit functon $b(\cdot )$ and use it in subsequent computations. We note that $\log{b(\cdot )}$ is always well-defined because eigenvalues are simple.
The Combinatorial Lemma implies the existence of
$$  c_4 M^2 \log{\left( \frac{M}{\varepsilon^2} \right)} \leq \ell \leq  2c_4 M^2 \log{\left( \frac{M}{\varepsilon^2} \right)},$$
where the constant $c_4$ depends only on $c_2$ and the validity of the Weyl scaling for $b(\cdot)$, for which
$$  \sum_{d+1 \leq n \leq c_2M}^{}{ \left\langle g, \phi_n \right\rangle \lambda_n^{-\ell}  \phi_n}  \qquad \mbox{behaves like a single frequency.}$$
We fix this value $\ell$ henceforth.
More precisely, we can deduce that
$$  \sum_{d+1 \leq n \leq c_2M}^{}{ \left| \left\langle g, \phi_n \right\rangle \lambda_n^{-\ell}  \right| }   \leq (1+\varepsilon) \left| \left\langle g, \phi_k \right\rangle \lambda_k^{-\ell}  \right| \leq \varepsilon $$
for some $d+1 \leq k \leq c_2M$. An application of the triangle inequality shows that
$$ \left\|   \left\langle g, \phi_k \right\rangle \lambda_k^{-\ell} \phi_k -  \sum_{d+1 \leq n \leq c_2M}^{}{  \left\langle g, \phi_n \right\rangle \lambda_n^{-\ell} \phi_n  }  \right\|_{L^{\infty}} \leq \varepsilon  \left(\max_{n \leq c_2 M}{\|\phi_n\|_{L^{\infty}}}\right)  \left\|   \left\langle g, \phi_k \right\rangle \lambda_k^{-\ell} \phi_k  \right\|_{L^{\infty}}.$$

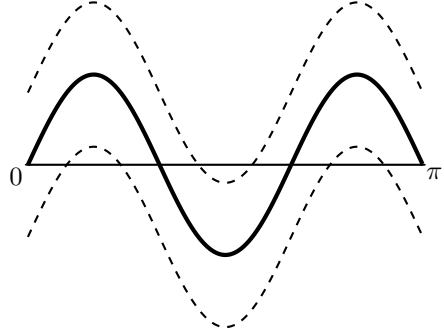
\begin{figure}[h!]
\begin{center}
\begin{tikzpicture}[xscale=3.5*1.5,yscale=0.08*1.5)]
\draw [ultra thick, domain=0:1, samples = 300] plot (\x, {(10*sin(3*pi*\x r)   )  }  );
\draw [dashed, thick, domain=0:1, samples = 300] plot (\x, {(10*sin(3*pi*\x r) - 8  )  }  );
\draw [dashed, thick, domain=0:1, samples = 300] plot (\x, {(10*sin(3*pi*\x r) + 8  )  }  );
\draw [thick, domain=0:1] plot (\x, {0}  );
\filldraw (0,0) ellipse (0.006cm and 0.048cm);
\node at (-0.03,-1.3) {0};
\filldraw (1,0) ellipse (0.006cm and 0.048cm);
\node at (1.03,-1) {$\pi$};
\end{tikzpicture}
\caption{The scale of local maxima/minima determines the natural size for an allowable $L^{\infty}-$bound such that a perturbation cannot have less roots.}
\end{center}
\end{figure}

We will now introduce $h(n)$ as the smallest extremal value (i.e. numerical value of a local maximum or local minimum) assumed by any of the function $\phi_1, \dots, \phi_n$ within one of their nodal domains.
If we are dealing with either Dirichlet or Neumann conditions, then this quantity simplifies to
$$ h(n) = \min_{1 \leq k \leq n} \left\{ \left| \phi_k(x)\right|: \phi_k'(x) = 0\right\}.$$
In the case of trigonometric functions, we simply have $h(n) = 1$. The WKB expansion suggests that we can generally expect $h(n) \sim 1$ with an implicit
constant only depending on the specific coefficients in the Sturm-Liouville problem. We now set $\varepsilon$ as
$$ \varepsilon =  \frac{1}{10} \frac{ h(c_2 M) }{ \max_{n \leq c_2 M}{\|\phi_n\|_{L^{\infty}}}}.$$
For trigonometric functions, this is merely $\varepsilon = 1/10$. This choice of parameters ensures that 
$$   \sum_{d+1 \leq n \leq c_2M}^{}{  \left\langle g, \phi_n \right\rangle \lambda_n^{-\ell} \phi_n  } \quad \mbox{has at least as many roots as} \quad \left\langle g, \phi_k \right\rangle \lambda_k^{-\ell} \phi_k ,$$
which, by the (weak) Sturm Oscillation Theorem, is known to have $k-1 \geq d$ roots. We now want to show that we can let the summation run up to $N+1$. We first establish that this main term is not too small.
Moreover, after possibly increasing $c_2$ by a factor of 2, we have
$$  \left\| \sum_{1 \leq n \leq c_2 M/2}^{}{ \left\langle g, \phi_n \right\rangle \phi_n} \right\|^2_{L^2} \geq \frac{\|g\|^2_{L^2}}{25}.$$
Orthogonality implies that either
$$  \left\| \sum_{1 \leq n \leq d}^{}{ \left\langle g, \phi_n \right\rangle \phi_n} \right\|^2_{L^2} \geq \frac{\|g\|^2_{L^2}}{50} \qquad \mbox{or} \qquad  \left\| \sum_{d+1 \leq n \leq c_2 M/2}^{}{ \left\langle g, \phi_n \right\rangle \phi_n} \right\|^2_{L^2} \geq \frac{\|g\|^2_{L^2}}{50}.$$
The first case is easily seen to imply the desired Theorem (indeed, it would be a much better result). We can thus assume that the second case applies.
Pigeonholing shows that there exists a constant $d+1 \leq k \leq c_2 M/2$ such that
$$   \left| \left\langle g, \phi_k \right\rangle\right|  \gtrsim \frac{\|g\|_{L^2}}{10\sqrt{c_2 M}}.$$
This shows that 
$$   \left\| \sum_{d+1 \leq n \leq c_2 M/2}^{}{ \left\langle g, \phi_n \right\rangle \lambda_n^{-\ell} \phi_n} \right\|_{L^2}  \geq \lambda_{c_2 M/2}^{-\ell}  \frac{\|g\|_{L^2}}{10\sqrt{c_2 M}}.$$
We see that 
\begin{align*}
h(c_2 M)   \left\| \sum_{d+1 \leq n \leq c_2 M/2}^{}{ \left\langle g, \phi_n \right\rangle \lambda_n^{-\ell} \phi_n} \right\|_{L^2} &\gtrsim   \lambda_{c_2 M/2}^{-\ell}  \frac{\|g\|_{L^2}}{10\sqrt{M}} \\
&\gg \left(\sup_{n \in \mathbb{N}}{ \|\phi_n\|_{L^{\infty}}}\right) \frac{\|g\|_{L^2}}{ (\lambda_{c_2M})^{ \ell - \frac{1}{4}}} \\
&=  \left\| \sum_{c_2M + 1\leq n \leq N+1}^{}{ \left\langle g, \phi_n \right\rangle \lambda_n^{-\ell}  \phi_n} \right\|_{L^{\infty}}.
\end{align*}
This shows that the tail is too small to affect the number of roots. The remainder of the argument is simple: if the first $d$ terms are also too small, then we would indeed have $k-1 \geq d$ roots which is a contradiction.
More precisely, if it were the case that
$$ \left\|  \sum_{n \leq d}^{}{ \left\langle g, \phi_n \right\rangle \lambda_n^{-\ell} \phi_n} \right\|_{L^{\infty}} \leq \sum_{n \leq d}{ \left| \left\langle g, \phi_n \right\rangle\right|} \leq \frac{\|g\|_{L^2}}{(\lambda_{c_2M})^{\ell}},$$
then this would imply that $g_{\ell}$ has at least $d$ roots in the interior of the interval. This is a contradiction and thus
$$ \sum_{n \leq d}{ \left| \left\langle g, \phi_n \right\rangle\right|} \geq  \frac{1}{\sup_{n \in \mathbb{N}}{ \|\phi_n\|_{L^{\infty}}}}  \frac{\|g\|_{L^2}}{(\lambda_{c_2M})^{\ell}} \qquad \mbox{as desired.}$$
Should $h(n)$ have some unexpectedly small values, then we increase the value of $\ell$ depending on that, incorporate it into $g(\kappa)$ and proceed in the same way.
\end{proof}

\end{document}